\documentclass[portrait,english,10pt]{article}

\usepackage[english]{babel}
\usepackage{textcomp}
\usepackage{eurosym}
\usepackage{amsthm}
\usepackage[intlimits]{amsmath}
\usepackage{amssymb}
\usepackage{amsfonts}
\usepackage{mathbbol}

\newtheorem{thm}{Theorem}[section]
\newtheorem{cor}[thm]{Corollary}

\theoremstyle{plain}

\newtheorem{thms}{Theorem}[section]
\newtheorem{cors}[thms]{Corollary}

\newtheorem{lems}[thms]{Lemma}
\newtheorem{props}[thms]{Proposition}

\theoremstyle{definition}
\newtheorem{defns}[thms]{Definition}

\numberwithin{equation}{section}

\theoremstyle{definition}

\numberwithin{equation}{section}

\newcommand{\EE}{\mathbb{E}}

\newcommand{\NN}{\mathbb{N}}

\newcommand{\PP}{\mathbb{P}}

\newcommand{\RR}{\mathbb{R}}

\newcommand{\dd}{\mathrm{d}}

\newcommand{\bB}{\mathcal{B}}
\newcommand{\cC}{\mathcal{C}}

\newcommand{\eE}{\mathcal{E}}

\newcommand{\lL}{\mathcal{L}}

\newcommand{\sS}{\mathcal{S}}

\newcommand{\vt}{\vartheta}
\newcommand{\al}{\alpha}

\newcommand{\e}{\varepsilon}
\newcommand{\la}{\lambda}

\newcommand{\pd}{\partial}
\newcommand{\mto}{\mapsto}
\newcommand{\ra}{\rightarrow}

\newcommand{\ti}{\tilde}

\newcommand{\vzv}{\Leftrightarrow}

\newcommand{\ds}{\displaystyle}

\newcommand{\ind}{\mathbf{1}}

\newcommand{\lqq}{\leqslant}
\newcommand{\gqq}{\geqslant}

\title{Asymptotic first exit times of the Chafee-Infante equation 
with small heavy-tailed L\'evy noise}

\author{ A. Debussche$^{1}$, Michael H\"ogele$^{2}$, Peter Imkeller$^{3}$ \\
\small$^{1}$ IRMAR - UMR 6625, ENS Cachan Bretagne, arnaud.debussche@bretagne.ens-cachan.fr\\
\small
\small$^{3}$ Universit\"at Potsdam, hoegele@math.uni-potsdam.de\\
\small
$^{3}$Humboldt-Universit\"at zu Berlin, imkeller@math.hu-berlin.de.
}

\begin{document}
\maketitle

\abstract{We study the first exit times form a reduced domain of attraction of a stable 
fixed of the Chafee-Infante equation when perturbed by a heavy tailed L\'evy noise with 
small intensity.
}



\section{Introduction}\label{chapter main results}

Energy balance models with random perturbations may provide crucial
probabilistic insight into paleoclimatological phenomena on a conceptual
level (see \cite{Ar01}, \cite{Im01}). Following the suggestion by
\cite{Dit99a} and \cite{Dit99b}, in \cite{IP08} the authors determine
asymptotic first exit times for one-dimensional heavy-tailed L\'evy
diffusions from reduced domains of attraction in the limit of small
intensity. Exponential moments not being available, the arguments leading
to these results do not employ large deviations methods, as opposed to
\cite{Go81}. \cite{IP08} shows that in contrast to the case of Gaussian
diffusions the expected first exit times are polynomial in terms of the
inverse intensity. In this article these finite dimensional results are
generalized to a class of reaction-diffusion equations, the prototype of
which is the Chafee-Infante equation.

Let $X^\e$ be the solution process of the stochastic Chafee-Infante
equation driven by $\e L$, an additive regularly varying L\'evy noise of
index $\alpha\in (0,2)$ at intensity $\e>0$. In this work we study the laws
of the asymptotic first exit times $\tau^\pm(\e)$ of $X^\e$ from a
(slightly reduced) domain of attraction of the deterministic Chafee-Infante
equation $u = X^0$ in the small noise limit $\e\to 0$. We show that there
exists a polynomial scale $\la^\pm(\e) \approx \e^\al$ linking the L\'evy
measure of $L$ and the domain of attraction of $u$, such that $\la^\pm(\e)
\tau^\pm(\e) \stackrel{d}{\ra} \bar \tau$, where $\bar \tau \sim EXP(1)$.
In particular $\EE[\tau^\pm(\e)]\approx \frac{1}{\e^\al}$ in the limit of
small $\e$.

This contrasts sharply with corresponding results in the case of Gaussian
perturbation \cite{FJL82}, where large deviations estimates in the spirit
of Freidlin and Ventsell are used to show exponential growth of first exit
times in the limit of small $\e$. Applied in a climatological context, the
Chafee-Infante equation is able to describe energy-balance based
reaction-diffusion equations, in which latitudinal heat transport is
possible, and states of the system can be seen as temperature distributions
on the interval between south and north pole. In this setting, our result
suggests a probabilistic interpretation of fast transitions between
different climate states corresponding to the stable equilibria observed in
ice core time series of temperature proxies of \cite{Cetal02}.

In the following sections we outline the partially tedious and complex
arguments needed to describe the asymptotic properties of the exit times.
Detailed proofs in particular of the more technical parts are given in the
forthcoming \cite{DHI10}.

\section{Preliminaries and the main result \label{preliminaries and notation}}

Let $H=H_0^1(0,1)$ be normed by $||u|| := |\nabla u|$ for $u\in H$,
where $|\cdot|$ is the norm in $L^2(0,1)$ and $\cC_0([0,1])$ the space of
continuous functions $u: [0,1]\ra \RR$ with $u(0) = u(1) =0$ and the supremum norm $|\cdot|_\infty$.
Since $|u| \lqq |u|_\infty \lqq ||u||$ for $u\in H$
we obtain the continuous injections $L^2(0,1) \hookrightarrow \cC_0(0,1) \hookrightarrow H$.
\noindent Denote by $M_0(H)$ the class of all Radon measures $\nu:\bB(H)\ra [0, \infty]$ satisfying
\[
\nu(A) <\infty \quad \vzv \quad A\in \bB(H), ~0 \notin \bar A.
\]
\noindent Let $(L(t))_{t\gqq 0}$ be a c\`adl\`ag version of a pure jump L\'evy process
in $H$ with a symmetric L\'evy measure $\nu\in M_0(H)$ on its
Borel $\sigma$--algebra $\bB(H)$ satisfying
\[
\int_H \min\{1, \|y\|^2\}\nu(\dd y) <\infty \qquad \mbox{ and } \qquad \nu(A) = \nu(-A), \quad A\in \bB(H), ~0 \notin \bar A,
\]
and which is regularly varying with index $\al = -\beta \in (0,2)$ and
limiting measure $\mu\in M_0(H)$. For a more comprehensive account we refer
to \cite{BGT87} and \cite{HL06}.

\noindent Fix $\pi^2 < \la \neq (\pi n)^2$ and $f(z) = -\la (z^3-z)$ for
$z\in \RR$. The object of study of this article is the behaviour of the
solution process $X^\e$ in $H$ of the following system for small $\e>0$.
For $x\in H$ consider

\begin{equation}\label{main sys}
\begin{array}{ccll}
\ds \frac{\pd}{\pd t}X^\e(t,\zeta) &=& \ds \frac{\pd^2 }{\pd \zeta^2} X^\e(t,\zeta) + f(X^\e(t,\zeta))  +
\e \dot L(t,\zeta), &  t>0,\;\zeta\in [0,1],\\[3mm]
X^\e(t,0) &=& X^\e(t,1) = 0, & t>0,\\[2mm]
X^\e(0,\zeta) &=& x(\zeta), & \zeta\in [0,1].
\end{array}
\end{equation}

\noindent We summarize some results for the solution $u(t;x) = X^0(t;x)$ of
the deterministic Chafee-Infante equation (ChI). It is well-known that the
solution flow $(t,x) \mto u(t;x)$ is continuous in $t$ and $x$ and defines
a dynamical system in $H$. Furthermore the solutions are extremely regular
for any positive time, i.e. $u(t) \in \cC^\infty(0,1)$ for $t>0$. The
attractor of (ChI) is explicitly known to be contained in the unit ball
with respect to the norm $|\cdot|_\infty$ (see for instance \cite{EFNT94},
Chapter 5.6).

\begin{props}\label{pointwise convergence}
For $\la>0$ denote by $\eE^\la$ the set of fixed points of (ChI).
Then for any $\la>0$ and initial value $x \in H$ there exists a stationary state
$\psi\in \eE^\la$ of the system (ChI) such that
\[
\lim_{t\ra\infty} u(t;x) = \psi.
\]
Furthermore if $\pi^2 < \la \neq (k\pi)^2, k\in \NN$, there are two stable
fixed points and all elements of $\eE^\la$ are hyperbolic. In addition, the
stable and the unstable manifolds of any unstable fixed point of $\eE^\la$
intersect transversally.
\end{props}

\noindent This relies on the fact that there is an energy functional, which
may serve as a Lyapunov function for the system. A proof of the first part
can be found in \cite{FJL82}, \cite{He83}, and of the second part in
\cite{He85}.

\begin{defns}
For $\la > \pi^2$ the solution of system (ChI) has two stable stationary
states denoted by $\phi^+$ and $\phi^-$. The full domains of attraction are
given by
\[
D^\pm := \{x\in H~|~\lim_{t\ra\infty} u(t;x) = \phi^\pm\}, \qquad \mbox{
and }\qquad D^\pm_0 := D^\pm-\phi^\pm,
\]
and the separatrix by
\[
\sS := H\setminus \left(D^+\cup D^-\right).
\]
\end{defns}

\noindent Due to the Morse-Smale property the separatrix is a closed
$\cC^1$-manifold without boundary in $H$ of codimension $1$ separating
$D^+$ from $D^-$, and containing all unstable fixed points. For more
refined results we refer to \cite{Ra01} and references therein.

\begin{defns}
Writing $B_\delta(x)$ for the ball of radius $\delta > 0$ in $H$
with respect to the $|\cdot|_\infty$--norm centered at $x$, denote
for $\delta_1, \delta_2, \delta_3\in (0,1)$
\begin{align}
D^\pm(\delta_1) := &\{x\in D^\pm~|\cup_{t\gqq 0} B_{\delta_1}(u(t;x)) \subset D^\pm\},\nonumber \\
D^\pm(\delta_1, \delta_2) := &\{x\in D^\pm~|\cup_{t\gqq 0}
 B_{\delta_2}(u(t;x)) \subset D^\pm(\delta_1)\},\nonumber\\
D^\pm(\delta_1, \delta_2, \delta_3) := &\{x\in D^\pm~|\cup_{t\gqq
0} B_{\delta_3}(u(t;x)) \subset D^\pm(\delta_1, \delta_2)\}.\label{reduced domains}
\end{align}
For $\gamma \in (0,1)$ the sets $\ti D^\pm(\e^\gamma) :=  D^\pm(\e^\gamma,
\e^{2\gamma})$ and $D^\pm(\e^\gamma, \e^{2\gamma}, \e^{2\gamma})$ will be
of particular importance. We define the reshifted domains of attraction by
\begin{align}
D^\pm_0(\delta_1):= &D^\pm(\delta_1) - \phi^\pm,\\
D_0^\pm(\delta_1, \delta_2) :=& D^\pm(\delta_1, \delta_2) - \phi^\pm,\\
D^\pm_0(\delta_1, \delta_2, \delta_3) :=& D^\pm(\delta_1,\delta_2,
\delta_3) - \phi^\pm, \end{align} and the following neighborhoods of the
separatrix $\sS$
\begin{align*}
\ti D^0(\e^\gamma) & := H \setminus \big(\ti D^+(\e^{\gamma})\cup ~\ti D^-(\e^\gamma)\big),\\
D^*_0(\e^\gamma) & := \big(D_0^\pm(\e)\setminus D_0(\e^\gamma, \e^{2\gamma})\big) + B_{\e^{2\gamma}}(0).
\end{align*}
\end{defns}

\noindent In \cite{DHI10} it is shown that the union over all
$\e>0$ for each of the sets $D^\pm(\e^\gamma)$, $\ti D^\pm(\e^\gamma)$
and $D^\pm(\e^\gamma, \e^{2\gamma}, \e^{2\gamma})$ exhausts $D^\pm$.
Furthermore $D^\pm(\e^\gamma)$ and $\ti D^\pm(\e^\gamma)$ are positively invariant under the deterministic solution flow,
and $\ti D^\pm(\e^\gamma) + B_{\e^{2\gamma}}(0) \subset ~D^\pm(\e^\gamma)$
and $D^\pm(\e^\gamma, \e^{2\gamma}, \e^{2\gamma}) + B_{\e^{2\gamma}}(0) \subset ~\ti D^\pm(\e^\gamma)$.

\begin{props}\label{logarithmic convergence time}
Given the Chafee-Infante parameter $\pi^2 < \la \neq (k\pi)^2$ for all $k\in \NN$
there exist a finite time $T_{rec} = T_{rec}(\la) >0$ and a constant $\kappa= \kappa(\la)>0$,
which satisfy the following.
For each $\gamma>0$ there is $\e_0 = \e_0(\gamma)>0$,
such that for all $0 <\e \lqq \e_0$, $T_{rec}+ \kappa\gamma |\ln\e|\lqq t$ and
$x\in D^{\pm}(\e^{\gamma})$
\[
|u(t;x)-\phi^\pm|_\infty \lqq (1/2)\e^{2\gamma}.
\]
\end{props}

\noindent This results relies on the hyperbolicity of the fixed points and
the fine dynamics of the deterministic solution flow. In \cite{DHI10} it is
proved in the stronger Hilbert space topology of $H$. The preceding theorem
follows then as a corollary.

\noindent We denote the jump increment of~$L$ at time $t\gqq 0$ by
$\Delta_t L := L(t) - L(t-)$, and decompose the process $L$ for $\rho\in
(0,1)$ and $\e>0$ in the following way. We call $\eta^\e$ the ``large
jump'' compound Poisson process with intensity $\beta_\e :=
\nu\left(\e^{-\rho} B_{1}^c(0)\right)$
and jump probability measure 
$\ds\nu(\cdot \cap \e^{-\rho} B_{1}^c(0))/\beta_\e$, 
and the complementary ``small jump'' process $\xi^\e:=L-\eta^\e.$ The
process $\xi^\e$ is a mean zero martingale in $H$ thanks to the symmetry of
$\nu$ with finite exponential moments. We define the jump times of
$\eta^\e$ as
\[
T_0 :=0, \qquad T_{k} := \inf\left\{t>T_{k-1}~\big|~\|\Delta_t L\| > \e^{-\rho} \right\},
\quad k\gqq 1,
\]
and the times between successive large jumps of $\eta^\e_t$ recursively as
$t_0 = 0$ and \mbox{$t_k := T_k-T_{k-1}$,} for $k\gqq 1.$ Their laws
$\lL(t_k)$ are exponential $EXP(\beta_\e)$. We shall denote the $k$-th
large jump by $W_0 = 0$ and \mbox{$W_k= \Delta_{T_k} L$} for $k\gqq 1.$

\begin{props}
For any mean zero $L^2(\PP; H)$-martingale $\xi = (\xi(t))_{t\gqq 0}$,
\mbox{$T>0$,} and initial value $x\in H$ equation (\ref{main sys}) driven
by $\e\xi$ instead of $\e L$ has a unique c\`adl\`ag mild solution
$(Y^\e(t;x))_{t\in[0,T]}$. The solution process $Y^\e$ induces a
homogeneous Markov family satisfying the Feller property.
\end{props}

\noindent A proof can be found in \cite{PZ07}, Chapter 10. By localization
this notion of solution is extended to the heavy-tailed process $L$. In
\cite{DHI10} this will be carried out in detail.

\begin{cors}
For $x\in H$ equation (\ref{main sys}) has a c\`adl\`ag mild solution $(X^\e(t;x))_{t\gqq 0}$,
which satisfies the strong Markov property.
\end{cors}

\begin{defns} 
For $\gamma\in (0,1)$, $\e>0$, and the c\`adl\`ag mild solution
$X^\e(\cdot;x)$  of (\ref{main sys}) with initial position \mbox{$x\in \ti
D^\pm(\e^\gamma)$} we define the \textit{first exit time from the reduced
domain of attraction}
\[
\tau^\pm_x(\e) :=\inf\{t>0~|~X^\e(t;x)\notin D^\pm(\e^\gamma)\}.
\]

\end{defns}

We now introduce the following two hypotheses, which will be required in
our main theorem. They are natural conditions on the regularly varying
L\'evy measure $\nu$ with respect to the underlying deterministic dynamics
in terms of its limit measure $\mu$. See \cite{HL06} for the relationship
between $\nu$ and $\mu$, and (\ref{lambda epsilon}) below for the
particular scaling function $\frac{1}{\epsilon}$ needed here.

\noindent\textbf{(H.1) Non-trivial transitions: } $\mu\left(\left(D^\pm_0\right)^c\right)>0.$

\noindent\textbf{(H.2) Non-degenerate limiting measure: }\textit{For
$\alpha\in (0,2)$ and $\Gamma>0$ according to Proposition \ref{small
deviations from the deterministic system} let
\begin{equation}\label{constants}
\hspace{-0.5cm}0 < \Theta < \frac{2-\al}{2\al}, \quad \rho\in (\frac{1}{2},
\frac{2-\al}{2-(1-\Theta)\alpha}),  \quad 0 < \gamma <
\frac{(2-\al)(1-\rho)- \Theta \alpha \rho}{2(\Gamma +2)}.
\end{equation}
\noindent For $k = \pm$ and $\eta>0$ there is $\e_0>0$ such that for all
$0<\e \lqq \e_0$}
\begin{equation}\label{measure hypothesis'}
\mu\left(H\setminus \left((D^+(\e^\gamma, \e^{2\gamma}, \e^{2\gamma})\cup
D^-(\e^\gamma, \e^{2\gamma}, \e^{2\gamma})) +
B_{\e^{2\gamma}}(0)\right)-\phi^k\right) < \eta.
\end{equation}

While (H.1) ensures that there actually are transitions also by ``large''
jumps with positive probability, (H.2) implies that the slow deterministic
dynamics close to the separatrix does not distort the generic exit scenario
of $X^\e$. For comparable finite dimensional situations with absolutely
continuous L\'evy measure $\nu\ll dx$ these hypotheses are always
satisfied.

\noindent For $\e>0$ we define the characteristic rate of the system
(\ref{main sys}) by
\begin{align}
\la^\pm(\e) := \nu\left(\frac{1}{\e} \left(D_0^{\pm}\right)^c\right). 
\end{align}

\noindent According to \cite{BGT87} and \cite{HL06} for $\nu$ chosen above
there is a slowly varying function $\ell_\nu = \ell: [0, \infty)\ra [0,
\infty)$ such that for all $\e>0$
\begin{align}\label{lambda epsilon}
\la^\pm(\e) =& ~\e^{\al}\;\ell (\frac{1}{\e})\;
\mu\left((D_0^\pm)^c\right), \quad \mbox{ and } \quad
~\beta_\e = ~\e^{\alpha\rho}\,\,\ell
(\frac{1}{\e^\rho})\,\,\mu\left(B_1^c(0)\right).
\end{align}

\noindent We may now state the main theorem.

\begin{thms}\label{first exit times}
Given the Chafee-Infante parameter $\pi^2 < \la \neq (k\pi)^2$ for all
$k\in \NN$, we suppose that Hypotheses (\textbf{H.1}) and (\textbf{H.2})
are satisfied. Then for any $\theta > -1$
\[
\lim_{\e \ra 0+} \EE\left[\sup_{x\in \ti D^\pm(\e^\gamma)}
\exp\left(-\theta \la^\pm(\e) \tau^\pm_x(\e)\right)\right] =
\frac{1}{1+\theta}.
\]
The supremum in the formula can be replaced by the infimum.

\end{thms}

The theorem states that in the limit of small $\e$, suitably renormalized
exit times from reduced domains of attraction have unit exponential laws.

\section{The Small Deviation of the Small Noise Solution}\label{chapter small deviations}

\noindent This section is devoted to a small deviations' estimate. It
quantifies the fact, that in the time interval between two adjacent large
jumps the solution of the Chafee-Infante equation perturbed by only the
small noise component deviates from the solution of the deterministic
equation by only a small $\e$-dependent quantity, with probability
converging to 1 in the small noise limit $\e\to 0.$ Define the stochastic
convolution $\xi^*$ with respect to the small jump part $\xi^\e$ by
$\xi^*(t) = \int_0^t S(t-s)\dd \xi^\e(s)$ for $t\gqq 0$ (see \cite{PZ07}).
In order to control the deviation for $Y^\e-u$ for small $\e>0$, we
decompose $Y^\e = u + \e\xi^* + R^\e$. By standard methods we obtain in
\cite{DHI10} the following lemmas.

\begin{lems} \label{noise control2} For $\rho \in \left(0,1\right)$, $\gamma >0$, \mbox{$p > 0$}
and $0 <\Theta< 1$ there are constants $C >0$ and $\e_0>0$ such that
for $0<\e \lqq \e_0$ and $T\gqq 0$
\[
\PP\big(\sup_{t\in [0,T]} \|\e\xi^*_t\| \gqq \e^{p} \big) \lqq C\;  T\;\e^{2-2p-(2-(1-\Theta)\al)\rho}.
\]
\end{lems}

\noindent Define for $T>0$, $\Gamma>0$ and $\gamma>0$ the small convolution event
\[
\eE_T(\e^{(\Gamma +2)\gamma}):= \{\sup_{r\in [0,T]} ||\e \xi^*(r)|| < \e^{(\Gamma +2)\gamma}\}\quad \e>0.
\]

\noindent By perturbation arguments, the stability of $\phi^\pm$, Proposition \ref{logarithmic convergence time} and Lemma \ref{noise control2}
we may estimate the remainder term $R^\e$ for small $\e$.

\begin{lems}\label{R control}
There is a constant $\Gamma>0$ such that for $\rho\in(1/2,1), \gamma>0$,
there exists $\e_0>0$ such that for $0< \e\lqq \e_0$, $T>0$, $x\in
D^\pm(\e^\gamma)$ on the event $\eE_T(\e^{(\Gamma+2)\gamma})$ we have the estimate
\begin{equation*}
\sup_{t\in [0,T]} |R^\e(t;x)|_\infty \lqq \frac{1}{4} \e^{2\gamma}.
\end{equation*}
\end{lems}

\noindent We next combine Proposition \ref{logarithmic convergence time},
Lemma \ref{noise control2} and Lemma \ref{R control}, to obtain the
following proposition on small deviations on deterministic time intervals.

\begin{props}\label{fixed horizon}
There is a constant $\Gamma>0$ such that for $0 < \al < 2$ given the
conditions
\[
0 < \Theta < \frac{2-\al}{\al}, \qquad \rho\in (1/2,
\frac{2-\al}{2-(1-\Theta)\alpha}),  \qquad 0 < \gamma <
\frac{(2-\al)(1-\rho)- \Theta \alpha \rho}{2(\Gamma +2)},
\]
there exist $\e_0>0$ and $C>0$ such that for any $T>0$, $0 <\e \lqq
\e_0$ and $x\in D^\pm(\e^\gamma)$
\begin{equation}\label{fixedhorizon}
\PP\big(\sup_{s\in [0,T]} |Y^\e(s;x) -u(s;x)|_\infty \gqq (1/2)
\e^{2\gamma} \big) \lqq C \;T\; \e^{2-2(\Gamma +2)\gamma -
(2-(1-\Theta)\al)\rho}.
\end{equation}
\end{props}

\noindent This can be generalized to the first jump time $T_1$ replacing
$T$.

\begin{props}\label{small deviations from the deterministic system} There is a constant
$\Gamma >0$ such that for $0<\alpha<2$ given the conditions
\[
0 < \Theta < \frac{2-\al}{\al}, \qquad \rho\in (1/2,
\frac{2-\al}{2-(1-\Theta)\alpha}), \qquad 0 < \gamma <
\frac{(2-\al)(1-\rho)- \Theta \alpha \rho}{2(\Gamma +2)},
\]
there exist constants $\vt = \vt(\Theta, \rho, \gamma, \al) >\al (1-\rho)$,
$C_\vt>0$ and $\e_0>0$, which satisfy for all $0<\e \lqq \e_0$
\[
\PP\big(\exists\;x\in  D^\pm(\e^\gamma):~\sup_{s\in [0,T_1]} |Y^\e(s;x) -u(s;x)|_\infty
\gqq (1/2) \e^{2\gamma} \big) \lqq C_\vt \e^{\vt}.
\]
\end{props}

\begin{proof}
Let $\Gamma>0$ large enough such that the hypotheses of Lemma \ref{R
control} are satisfied. Then with the given constants there exist constants
$C_\theta>0$ and $\e_0>0$ such that for $0< \e \lqq \e_0$
\begin{multline*}
\PP\big(\exists\; x\in D^\pm(\e^\gamma):~\sup_{s\in [0,T_1]} |Y^\e(s;x) -u(s;x)|_\infty \gqq (1/2) \e^{2\gamma} \big) \\
\shoveleft \lqq  \int_0^\infty \PP\big(\exists\; x\in D^\pm(\e^\gamma):~\sup_{s\in [0,t]} |Y^\e(s;x) -u(s;x)|_\infty \gqq (1/2) \e^{2\gamma} \big) \beta_\e e^{- \beta_\e t} ~\dd t\\
\lqq  C_\theta \;\e^{2-2(\Gamma +2)\gamma - (2-(1-\Theta)\al)\rho-\al\rho}.
\end{multline*}

\noindent Fix $\vartheta = 2-2(\Gamma +2)\gamma -
(2-(1-\Theta)\al)\rho-\al\rho$. One checks that $\vartheta>\alpha (1-\rho)$.
\end{proof}

\noindent For $x\in D^\pm(\e^\gamma)$ define the small perturbation event
\begin{equation*}
E_x := \{\sup_{s\in [0,T_1]} |Y^\e(s;x)-u(s;x)|_\infty \lqq (1/2) \e^{2\gamma}\}.
\end{equation*}
\begin{cor}\label{Event E^c}
Given the assumptions of Proposition \ref{small deviations from the
deterministic system}
there is a constant \mbox{$\vartheta = \vartheta(\al, \Theta, \gamma, \rho)$} with
\mbox{$\vartheta > \al(1-\rho)$,} \mbox{$C_\vartheta>0$,} and $\e_0>0$ such that for all $0 < \e\lqq \e_0$
\begin{equation*}
\EE\left[\sup_{x\in D^\pm(\e^\gamma)}\ind(E^c_x)\right]\lqq C_\vartheta \e^{\vartheta}.
\end{equation*}
\end{cor}

\begin{cor}\label{T1 Event E^c}
Let $C>0,$ and let the assumptions of Proposition \ref{small deviations
from the deterministic system} be satisfied. Then there is a constant $\e_0>0$ such that for
all $0 < \e\lqq \e_0$, $\theta>-1$
\begin{equation}
\EE\left[e^{-\theta \la^\pm(\e) T_1} \sup_{x\in
D^\pm(\e^\gamma)}\ind(E^c_x)\right]\lqq C \left(\frac{\beta_\e}{\beta_\e +
\theta \la^\pm(\e)}\right) \frac{\la^\pm(\e)}{\beta_\e}.
\end{equation}
\end{cor}

\section{Asymptotic first exit times}\label{chapter exit times}

In this section we derive estimates on exit events which then enable us to
obtain upper and lower bounds for the Laplace transform of the exit times
in the small noise limit.

\subsection{Estimates of Exit Events by Large Jump and Perturbation Events \label{Event notation}}

To this end, in this subsection we first estimate exit events of $X^\e$ by
large jump exits on the one hand, and small deviations on the other hand.
Denote the shift by time $t$ on the space of trajectories by $\theta_t,
t\gqq 0$. For any $k\in \NN$, $t\in [0,t_k]$, $x\in H$ we have
\begin{equation}\label{def small jumps solution}
X^\e(t+T_{k-1}; x) = Y^\e(t;X^\e(0;x))\circ \theta_{T_{k-1}} + \e
W_k \mathbf{1}\{t = t_k\}.
\end{equation}

\noindent In the following two lemmas we estimate certain events connecting
the behavior of $X^\e$ in the domains of the type $D^\pm(\e^{\gamma})$ with
the large jumps $\eta^\e$ in the reshifted domains of the type
$D_0^\pm(\e^{\gamma})$. We introduce for $\e>0$ and $x\in \ti
D^\pm(\e^{\gamma})$ the events
\begin{align}\label{def events}
A_x :=& \{Y^\e(s;x)\in D^\pm(\e^\gamma) \mbox{ for } s\in [0, T_1] \mbox{ and }
Y^\e(T_1;x)+\e W_1 \in D^\pm(\e^\gamma)\}, \nonumber\\
B_x :=& \{Y^\e(s;x)\in D^\pm(\e^\gamma) \mbox{ for } s\in [0, T_1] \mbox{ and }
Y^\e(T_1;x)+\e W_1 \notin D^\pm(\e^\gamma)\}, \nonumber\\
C_x :=& \{Y^\e(s;x)\in D^\pm(\e^\gamma) \mbox{ f. } s\in [0, T_1] \mbox{ a. }
Y^\e(T_1;x)+\e W_1 \in D^\pm(\e^\gamma)\setminus \ti D^\pm(\e^\gamma)\}, \nonumber\\
A^{-}_x :=& \{Y^\e(s;x)\in D^\pm(\e^\gamma) \mbox{ for } s\in [0, T_1]
\mbox{ and } Y^\e(T_1;x)+\e W_1 \in \ti D^\pm(\e^{\gamma})\}.
\end{align}

\noindent We exploit the definitions of the reduced
domains of attraction in order to obtain estimates of
solution path events by events only depending on the driving noise.

\begin{lems}[Partial estimates of the major events] \label{connect path and noise} Let $T_{rec}, \kappa>0$ be given by Proposition \ref{logarithmic convergence time} and assume that Hypotheses (H.1) and (H.2)
are satisfied. For $\rho \in \left(\frac{1}{2},1\right)$, $\gamma\in
(0,1-\rho)$ there exists $\e_0>0$ so that the following inequalities hold
true for all $0<\e \lqq \e_0$ and $x\in  D^\pm(\e^\gamma)$
\begin{align}
i) & ~\ind(A_x) \ind(E_x) \ind\{T_1 \gqq T_{rec} + \kappa\gamma |\ln \e|\} \lqq \mathbf{1}\{\e W_1 \in D_0^\pm\},&\label{Ay}\\
ii) & ~\ind(B_x) \ind(E_x) \ind\{T_1 \gqq T_{rec} + \kappa\gamma |\ln\e|\} \lqq \ind \{\e W_1 \notin D^\pm_0(\e^\gamma, \e^{2\gamma})\},&\label{By}\\
iii) & ~\ind(C_x) \ind(E_x) \ind\{T_1 \gqq T_{rec} + \kappa\gamma |\ln\e|\} \lqq \ind \{\e W_1 \in D^*_0(\e^\gamma)\}.&\label{Cy}
\end{align}
\noindent Additionally, for $x\in  D^\pm(\e^\gamma)$ we have
\begin{align}
iv) & ~\ind(B_x) \ind(E_x) \ind\{\|\e W_1\| \lqq (1/2)\e^{2\gamma}\} \ind\{T_1 > T_{rec} + \kappa\gamma |\ln\e|\} = 0,&\\
v) & ~\ind(C_x) \ind(E_x) \ind\{\|\e W_1\|  \lqq (1/2) \e^{2\gamma}\} \ind\{T_1 \gqq T_{rec}+ \kappa\gamma |\ln\e|\} = 0.&
\end{align}
\noindent In the opposite sense for $x\in \ti D^\pm(\e^\gamma)$
\begin{align}
vi) & ~\ind(E_x)\ind\{T_1\gqq T_{rec} + \kappa\gamma |\ln \e|\}
\ind\{\e W_1 \notin D^\pm_0\} \lqq  \ind(B_x), &  \label{Bym}\\
vii) & ~\mathbf{1}(E_x) \ind\{T_1\gqq T_{rec} + \kappa\gamma |\ln \e|\} \ind\{\e W_1\in D^\pm_0(\e^{\gamma}, \e^{2\gamma}, \e^{2\gamma})\} \lqq \ind(A^-_x). &\label{Aym}
\end{align}
\end{lems}

\noindent With the help of Lemma \ref{connect path and noise} we can show the following crucial estimates.
\begin{lems} [Full estimates of the major events] \label{connect path and noise 2}
Let $T_{rec}, \kappa>0$ be given by Proposition \ref{logarithmic
convergence time} and Hypotheses (H.1) and (H.2) be satisfied. For $\rho
\in \left(\frac{1}{2},1\right)$, $\gamma\in (0,1-\rho)$ there exists
$\e_0>0$ such that the following inequalities hold true for all $0<\e \lqq
\e_0, \kappa>0$ and $x\in D^\pm(\e^\gamma)$
\begin{align*}
ix) & ~\ind(A_x) &\lqq & \ind\{\e W_1\in D^\pm_0\}
+ \ind\{\|\e W_1\| >\frac{1}{2} \e^{2\gamma}\} \ind\{T_1<T_{rec} + \kappa\gamma |\ln \e|\} + \ind(E_x^c),& \\
x) & ~\ind(B_x) &\lqq & \ind \{\e W_1 \notin D^\pm_0(\e^\gamma,
\e^{2\gamma})\} + \ind\{T_1 < T_{rec} + \kappa\gamma |\ln \e|\} +\ind(E^c_x),
\end{align*}
\begin{align*}
xi) &~\sup_{y\in \ti D^\pm(\e^{\gamma})}\ind\{ Y^{\e}(s;y)\notin D^\pm(\e^\gamma) \mbox{ for some }s\in (0, T_1) \}
\lqq  \sup_{y\in \ti D^\pm(\e^{\gamma})} \ind(E^c_y),& \\
xii) & ~\ind(A_x) \ind\{Y^{\e}(s;X^\e(0,x))\circ\theta_{T_1}\notin D^\pm(\e^\gamma) \mbox{ for some }s\in(0, T_1)\} & \\
&\lqq \ind \left\{\e W_1 \in D^*_0(\e^\gamma) \right\} + \ind\{T_1
< T_{rec} + \kappa\gamma |\ln \e|\} + \sup_{y\in \ti
D^\pm(\e^\gamma)} \ind(E_y^c)\circ \theta_{T_1} + \; \ind(E^c_x).&
\end{align*}
\noindent In the opposite sense for $x\in \ti D^\pm(\e^\gamma)$
\begin{align*}
xiii) & ~\ind(A_x^-) \gqq \ind\{\e W_1\in D_0^\pm(\e^\gamma,
\e^{2\gamma}, \e^{2\gamma})\} -\ind\{T_1 < T_{rec}+\kappa\gamma |\ln\e|\}- 2\;\ind(E^c_x),&\\
xiv) & ~\ind(B_x) \gqq \ind\{\e W_1\notin D^\pm_0\} (1-\ind\{T_1 < T_{rec} + \kappa\gamma|\ln\e|\})-\ind(E^c_x).&
\end{align*}
\end{lems}
\noindent The next lemma ensures that after having relaxed to
$B_{\e^{2\gamma}}(\phi^\pm)$ the solution $X^\e$ jumps close to the
separatrix only with negligible probability for $\e \ra 0+$.

\begin{lems}[Asymptotic behavior of large jump events]\label{asymptotics of large jump events}
Let Hypotheses (H.1) and (H.2) be satisfied and $1/2 < \rho < 1-2\gamma$.
Then for any $C>0$ there is $\e_0 =\e_0(C) >0$ such that for all $0<\e \lqq
\e_0$
\begin{align*}
I) & \left(\frac{\ds \mu\left((D^\pm_0)^c\right)}{\ds \mu(B_1^c(0))} -C\right) \e^{\al (1-\rho)}
\lqq  \frac{\ds \la^\pm(\e)}{\ds \beta_\e} \lqq \left(\frac{\ds \mu((D_0^\pm)^c)}{\ds \mu(B_1^c(0))}
+C\right) \e^{\al (1-\rho)}, &\\
II) & ~\PP\left(\|\e W_1\| > (1/2) \e^{2\gamma}\right) \lqq 4 \e^{\al(1-\rho -2\gamma)},&\\
III) & ~\PP\left(\e W_1 \in (\ti D^\pm_0(\e^{\gamma}))^c \right) \lqq \left(1+C\right) \frac{\ds
\la^\pm(\e)}{\ds \beta_\e},& \\
IV) & ~\PP\left(\e W_1 \in D^*_0(\e^\gamma)\right) \lqq C \frac{\ds \la^\pm(\e)}{\ds \beta_\e},&\\
V) & ~\PP(\e W_1 \in D_0^c(\e^\gamma, \e^{2\gamma}, \e^{2\gamma}))
\lqq  (1+C) \frac{\ds \la^\pm(\e)}{\ds \beta_\e}.&
\end{align*}
\end{lems}

\noindent A detailed proof is given in \cite{DHI10}.

\subsection{Asymptotic Exit Times from Reduced Domains of Attraction}

We next exploit the estimates obtained in the previous subsection and
combine them with the small deviations result of section 3, to identify the
exit times from the reduced domains of attraction with large jumps from
small neighborhoods of the stable equilibria that are large enough to cross
the separatrix.

\begin{props}[The upper estimate\label{the upper estimate}]
Let (H.1) and (H.2) be satisfied. Then for all
$\theta > -1$ and $C\in (0,1+\theta)$ there is
$\e_0=\e_0(\theta)>0$ such that for $0<\e \lqq \e_0$
\[
\EE\left[\sup_{x\in \ti D^\pm(\e^\gamma)} \exp\left(-\theta \la^\pm(\e) \tau_x^\pm(\e)\right)\right]
\lqq \frac{1+C}{1+\theta -C}.
\]
\end{props}

\begin{proof} By (H.2) $\Gamma>0$ can be chosen large enough to
fulfill the hypotheses of Proposition \ref{small deviations from the
deterministic system}. Let $C>0$ be given. We drop the superscript $\pm$.
Since the jumps of the noise process $L$ exceed any fixed barrier
$\PP$-a.s., $\tau_x (\e)$ is $\PP$-a.s. finite. Therefore we may rewrite
the Laplace transform of $\tau_x(\e)$ for $\e>0$, giving
\begin{multline}
\EE\left[\sup_{x\in \ti D(\e^\gamma)} e^{-\theta \la(\e) \tau_x (\e)}\right] = \sum_{k=1}^\infty \bigg(\EE \bigg[e^{-\theta \la(\e) T_k} \sup_{x\in \ti
D(\e^\gamma)}\ind\{\tau_x (\e) = T_k \}\bigg]\label{central sum}\\
+\EE\bigg[\sup_{x\in \ti D(\e^\gamma)} e^{-\theta \la(\e) \tau_x(\e)}
\ind\{\tau_x (\e) \in (T_{k-1}, T_k) \}\bigg]\bigg) = I_1 + I_2.
\end{multline}

\noindent Using the strong Markov property, the independence and stationarity of the increments
of the large jumps $W_i$ we obtain for $k\gqq 1$
\begin{multline*}
\EE\left[ e^{-\theta \la(\e) T_k} \sup_{x\in \ti D(\e^\gamma)}\ind\{\tau_x(\e) = T_k \}\right]\\
\lqq \left(\EE \left[e^{-\theta \la(\e) T_1} \sup_{y\in
D(\e^\gamma)} \ind\left(A_y\right)\right]\right)^{k-1} \EE
\left[e^{-\theta \la(\e) T_1}\sup_{y\in D(\e^\gamma)}
\ind\left(B_y\right)\right].
\end{multline*}

\noindent In the subsequent Claims 1-4 we estimate the preceding factors
with the help of Lemma~\ref{connect path and noise 2}.
\paragraph{Claim 1: } There exists $\e_0>0$ such that for all $0 < \e \lqq \e_0$
\begin{equation*}
\EE_x\left[e^{-\theta \la(\e) T_1} \sup_{y\in D(\e^\gamma)}
\ind(A_y) \right]\lqq \frac{\beta_\e}{\beta_\e +\theta \la(\e)}
\left(1- \frac{\la(\e)}{\beta_\e}(1-C/5)\right).
\end{equation*}

\noindent In fact: in the inequality of Lemma \ref{connect path and noise
2} $ix)$ we can pass to the supremum in $y\in D(\e^\gamma)$, and integrate
to obtain, using the independence of $(W_i)_{i \in \NN}$ and $(T_i)_{i\in
\NN}$
\begin{multline*}
\EE \left[e^{-\theta \la(\e) T_1} \sup_{y\in D(\e^{\gamma})} \ind( A_y)\right] \\
\lqq \EE\left[e^{-\theta \la(\e) T_1} \ind\{T_1 < T_{rec} + \kappa\gamma |\ln \e|\}\right] \PP\left(\e \|W_1\| > (1/2) \e^{2\gamma}\right) \\
+\EE \left[e^{-\theta \la(\e) T_1}\right] \PP\left(\e W_1 \in D_0 \right) + \EE\left[e^{-\theta \la(\e) T_1}\sup_{y\in D(\e^{\gamma})}\ind(E_y^c)\right] \\
=: K_1 K_2 + K_3 K_4 + K_5.
\end{multline*}

\noindent The terms $K_1$, $K_3$ and $K_4$ can be calculated explicitly,
for $K_2$ we apply Lemma~\ref{asymptotics of large jump events} $II)$. For
$K_5$ we use Corollary~\ref{T1 Event E^c} and Lemma~\ref{asymptotics of
large jump events} $I)$ ensuring that there is $\e_0$ so that we have for
$0<\e\lqq \e_0$
\begin{equation}\label{K5}
K_5 \lqq C/10\,\frac{\beta_\e}{\beta_\e + \theta \la(\e)}
\frac{\la(\e)}{\beta_\e}.
\end{equation}

\paragraph{Claim 2: } There is $\e_0>0$ such that for all $0<\e \lqq \e_0$
\begin{equation*}
\EE \left[e^{-\theta \la(\e) T_1}\sup_{y\in D(\e^\gamma)} \ind\left(B(y)\right)\right] \lqq  (1+C)\frac{\beta_\e}{\beta_\e + \theta \la(\e)} \frac{\la(\e)}{\beta_\e}.
\end{equation*}

\noindent Indeed, in a similar manner and with the help of Lemma
\ref{connect path and noise 2} $x)$ and Lemma \ref{asymptotics of large
jump events} $III)$ we obtain that there is $\e_0>0$ such that for all
$0<\e \lqq \e_0$
\begin{equation*}
\EE \left[e^{-\theta \la(\e) T_1}\sup_{y\in D(\e^\gamma)} \ind\left(B(y)\right)\right] \lqq  (1+C)\frac{\beta_\e}{\beta_\e + \theta \la(\e)} \frac{\la(\e)}{\beta_\e}.
\end{equation*}

\noindent In order to treat the summands of the second sum of (\ref{central
sum}) we have to distinguish the cases $\theta \gqq 0$ and $\theta \in (-1,
0)$, as well as $k\ =1$ and $k\gqq 2$. Let us first discuss the case
$\theta\gqq 0.$

\paragraph{Claim 3: } There is $\e_0>0$ such that for all $0<\e \lqq \e_0$
\begin{equation*}
\EE\left[ \sup_{x\in\ti D(\e^\gamma)} e^{-\theta \la(\e) \tau_x(\e)}
\ind\{\tau_x(\e) \in (0, T_1) \}\right] \lqq
 C/5\,\left(\frac{\beta_\e}{\beta_\e + \theta \la(\e)}\right)
\frac{\la(\e)}{\beta_\e}.
\end{equation*}

This statement is proved by means of Lemma \ref{connect path and noise 2}
$xi)$ and Corollary \ref{Event E^c}.

\paragraph{Claim 4: } There exists $\e_0>0$ such that for any $k\gqq 2$

\begin{multline*}
\EE\left[\sup_{x\in \ti D(\e^\gamma)} e^{-\theta \la(\e) \tau_x(\e)} \ind\{\tau_x(\e) \in (T_{k-1}, T_k) \}\right]\\
\lqq \left(\frac{\beta_\e}{\beta_\e +\theta \la(\e)} \left(1-
\frac{\la(\e)}{\beta_\e}\left(1-C/5\right)\right)\right)^{k-2} C/5
\frac{\beta_\e}{\beta_\e+\theta \la(\e)} \frac{\la(\e)}{\beta_\e}.
\end{multline*}

\noindent To show this, we use the strong Markov property and Lemma
\ref{connect path and noise 2} $xii)$, as in the estimate for the first
summand to get for $k\gqq 2$ and $\theta \gqq 0$
\begin{multline}
\EE\left[\sup_{x\in \ti D(\e^\gamma)} e^{-\theta \la(\e) T_{k-1}} \ind\{\tau_x(\e)\in (T_{k-1}, T_{k})\}\right] \\
\lqq \left(\EE\left[ e^{-\theta\la(\e) T_1} \sup_{y\in D(\e^\gamma)}
\ind(A_y)\right]\right)^{k-2} \left(K_3 K_9 +K_1 + 2 K_5\right).
\end{multline}

\noindent Lemma~\ref{connect path and noise 2} $xii)$ and
Lemma~\ref{asymptotics of large jump events} $IV)$ provide the existence of
$\e_0>0$ such that for $0 < \e \lqq \e_0$
\begin{equation*}
K_9 = \PP\left(\e W_1 \in D^*_0(\e^\gamma)\right) \lqq C/20 \frac{\la(\e)}{\beta_\e}.
\end{equation*}
\noindent It remains to discuss the case $\theta \in (-1,0)$ in a similar
way. This is detailed in \cite{DHI10}.

Combining Claims 1-4 we finally find an $\e_0>0$ such that for
(\ref{central sum}) and all $0<\e \lqq \e_0$

\begin{multline*}
\EE\left[\sup_{x\in \ti D(\e^\gamma)} e^{-\theta \la(\e) \tau_x (\e)}\right]\\
\lqq \left(1+(2/5)C\right) \frac{\la(\e)}{\beta_\e} \frac{\beta_\e}{\beta_\e + \theta \la(\e)} \sum_{k=0}^\infty \left(\frac{\beta_\e}{\beta_\e + \theta \la(\e)} \left(1-\frac{\la(\e)}{\beta_\e} (1-C/5)\right)\right)^{k}\\
\lqq\frac{1+C}{\theta +(1-C)}.
\end{multline*}

\noindent The series converges if and only if $C< \theta +1$.
\end{proof}

\begin{props}[The lower estimate\label{the lower estimate}]
Assume that Hypotheses (H.1) and (H.2) are satisfied. Then for
all $\theta > -1$ and $C\in (0,1+\theta)$
there is $\e_0=\e_0(\theta)>0$ such that for all $0<\e \lqq \e_0$
\[
\EE\left[\inf_{x\in \ti D^\pm(\e^\gamma)} \exp\left(-\theta \la^\pm(\e) \tau_x^\pm(\e)\right)\right]
\gqq \frac{1+C}{1+\theta -C}.
\]
\end{props}

\begin{proof} Again we omit the superscript $\pm$ and fix $\Gamma>0$ large enough due to (H.2).
Omitting the term $I_2$ in equation (\ref{central sum}), we obtain the
estimate

\begin{multline}\label{reduced central sum}
\EE\left[\inf_{x\in \ti D(\e^\gamma)} e^{-\theta \la(\e) \tau_x(\e)}\right]\\
\gqq  \sum_{k=1}^\infty \left(\EE \left[e^{-\theta \la(\e) T_1} \inf_{y\in \ti D(\e^\gamma)}
\ind(A^-_y)\right]\right)^{k-1} \EE \left[e^{-\theta \la(\e) T_1}\inf_{y\in \ti D(\e^\gamma)} \ind(B_y)\right].
\end{multline}

\noindent We treat the terms appearing in (\ref{reduced central sum}) in a
similar way as for the upper estimate.
\paragraph{Claim 1: } There is $\e_0>0$ such that for all $0 < \e \lqq \e_0$
\begin{equation*}
\EE \left[e^{-\theta \la(\e) T_1} \inf_{x\in \ti D(\e^\gamma)}\ind(A^-_x)\right] \gqq \frac{\beta_\e}{\beta_\e +\theta \la(\e)}\left(1 - (1+C)\frac{\la(\e)}{\beta_\e}\right).
\end{equation*}
\noindent To prove this, we apply Lemma \ref{connect path and noise}
$xiii)$, take the infimum over $y\in \ti D(\e^\gamma)$ and integrate to get
\begin{multline*}
\EE\left[e^{-\theta \la(\e) T_1} \inf_{y\in \ti D(\e^\gamma)} \ind(A^-_y)\right] \\
= K_3 \left(1-\PP(W_1 \in (1/\e) D^c_0(\e^\gamma, \e^{2\gamma}, \e^{2\gamma}))\right) - K_1 - 2 K_5,\\
\end{multline*}
where $K_1, K_3, K_5$ have the same meaning as in the proof of Proposition
\ref{the upper estimate} and are treated identically.

\noindent By Lemma \ref{asymptotics of large jump events} $V)$ there exists
$\e_0>0$ such that for $0< \e \lqq \e_0$
\[
\PP(\e W_1 \in D_0^c(\e^\gamma, \e^{2\gamma}, \e^{2\gamma})) \lqq (1+C/5) \frac{\la(\e)}{\beta_\e}.
\]

\paragraph{Claim 2: } There is $\e_0>0$ such that for $0<\e \lqq \e_0$
\begin{equation*}
\EE\left[e^{-\theta \la(\e) T_1} \inf_{y\in \ti D_0(\e^{\gamma})}\ind(B_y)\right]
\gqq  \frac{\beta_\e}{\theta \la(\e) +\beta_\e} \left((1-C)\frac{\la(\e)}{\beta_\e}\right).
\end{equation*}

\noindent Here we exploit Lemma \ref{connect path and noise 2} $xiv)$.
Finally combining Claim~1 and Claim~2 we obtain
\begin{multline*}
\EE\left[\inf_{x\in \ti D(\e^\gamma)} e^{-\theta \la(\e) \tau_x(\e)}\right] \geq \\
\gqq \sum_{k=1}^\infty \left(\frac{\beta_\e}{\beta_\e +\theta \la(\e)}\left(1 - (1+C)\frac{\la(\e)}{\beta_\e}\right)\right)^{k-1} \frac{\beta_\e}{\theta \la(\e) +\beta_\e} \left((1-C)\frac{\la(\e)}{\beta_\e}\right) \\
= \frac{\la(\e) (1-C)}{\theta \la(\e) - (1+C) \la(\e)} = \frac{1-C}{\theta +1 +C}.
\end{multline*}
The series converges if and only if $-(1+C) < \theta$. \end{proof}


\end{document}